\documentclass[12pt]{amsart}
\usepackage{cite}
\usepackage{amsmath}
\usepackage{amsthm}
\usepackage{amssymb}
\usepackage{amsfonts}
\usepackage{array}
\usepackage{fullpage}
\usepackage{tikz-cd}
\usepackage{enumerate}
\usepackage{multicol}
\usepackage{geometry}
\usepackage{verbatim}
\usepackage{stmaryrd}
\usetikzlibrary{shapes.geometric, arrows}
\theoremstyle{plain}
\newtheorem*{theorem*}{Theorem}
\newtheorem*{principle*}{Principle}
\newtheorem{theorem}{Theorem}
\newtheorem{lemma}[theorem]{Lemma}
\newtheorem{proposition}[theorem]{Proposition}
\newtheorem{corollary}[theorem]{Corollary}

\newtheorem{definition}[theorem]{Definition}
\newtheorem*{lemma*}{Lemma}
\newtheorem*{proposition*}{Proposition}
\newtheorem*{corollary*}{Corollary}
\newtheorem*{conjecture*}{Conjecture}
\newtheorem*{definition*}{Definition}
\theoremstyle{remark}
\newtheorem{remark}[theorem]{Remark}

\theoremstyle{definition}

\tikzstyle{startstop} = [rectangle, rounded corners, minimum width=3cm, minimum height=1cm, text centered, text width= 4cm, draw=black]
\tikzstyle{io} = [rectangle, minimum width=3cm, minimum height=1cm, text centered, text width=3cm, draw=black]
\tikzstyle{process} = [rectangle, minimum width=3cm, minimum height=1cm, text centered, text width=4cm, draw=black]
\tikzstyle{decision} = [rectangle, minimum width=3cm, minimum height=1cm, text centered, text width=4cm, draw=black]
\tikzstyle{arrow} = [thick,->,>=stealth]

\def\dim{\operatorname{dim}}

\author{David Wen}
\title{Moduli of Canonical Surfaces of General Type with $K_S^2 = 1$, $p_g = 2$}
  \address{David Wen, National Center for Theoretical Sciences, No. 1 Sec. 4 Roosevelt Rd., National Taiwan University , Taipei, 106, Taiwan}
  \email{dwen@ncts.ntu.edu.tw}
 
\begin{document}

\newcommand{\bigslant}[2]{{\raisebox{.2em}{$#1$}\left/\raisebox{-.2em}{$#2$}\right.}}

\thanks{}

\begin{abstract} 
We study the moduli space of minimal surfaces of general type with $K_S^2 = 1$ and $p_g = 2$ and show that it is irreducible, has dimension $28$ and admits a compactification which is unirational. 
\end{abstract}

\maketitle

\section{Introduction}

The moduli of surfaces of general type is a vast wilderness, which becomes more manageable to study after fixing a few invariants. In \cite{Gieseker77}, Gieseker showed that for minimal surface of general type, $S$, there exists a quasiprojective course moduli space, and from \cite{Bombieri73}, by fixing invariants $K_S^2$, $p_g(S)$ and $q(S)$, the course moduli space, $M_{(K_S^2, p_g(S), q(S))}$, has finitely many components.

Some natural questions to ask about $M_{(K_S^2, p_g(S), q(S))}$ are: Is it irreducible? Is it connected? What is the dimension? The various results studying these questions have been carried out case by case and references can be made to \cite{BCP06} for some of the cases where the answers are known. 

In this paper, we consider the case of $(K_S^2, p_g(S), q(S)) = (1,2,0)$ and show the following:
\begin{theorem}[Theorem \ref{main}]
$M_{(1,2,0)}$ is irreducible, has dimension $28$ and has a projective compactification $\overline{M}_{(1,2,0)}$, via a GIT quotients, that is unirational.
\end{theorem}
Understanding this moduli space allows for a more in depth study fibrations whose general fiber is a surface of type $(K_S^2, p_g) = (1,2)$. This would provide explicit examples and further means to study general type varieties that are fibered by such surfaces.

The paper is structured as follows. Section 2 reviews over background material of minimal surfaces with $K_S^2 = 1$ and $p_g = 2$ and discusses related moduli problems to the the moduli space of such surfaces. Section 3 sets up a parameter space of isomophism classes of surfaces of type $K_S^2 = 1$ and $p_g = 2$ and show that there is a linearized $SL(2,\mathbb{C})$ action upon the parameter space. Section 4 concludes by giving an anylysis of the (semi-)stable locus of the action given in the previous section and shows that points corresponding to canonical surfaces are stable. The main theorem then follows from standard results in Geometric Invariant Theory. 

\section{Background and Related Problems}

\subsection{Explicit description of Canonical Surface of $K_S^2 = 1$ and $p_g(S) = 2$}
Minimal surfaces of type $K_S^2 = 1$ and $p_g = 2$ are well known as the exception to Bombieri's theorem, \cite[Theorem 1.1]{Catanese87}, where it's third and fourth canonical map are not birational. This is due to the slow growth of the canonical ring and, as a result, leads to the following explicit description of it's canonical model. 

\begin{proposition}[{\cite[Example 1.3]{Catanese87}} ]
\label{form1}
Let $S$ be a minimal surface of general type with $K_S^2 = 1$ and $p_g = 2$, then the canonical model of $S$ can be realized as a hypersurface in $\mathbb{P}(1,1,2,5)$ defined by the degree 10 weighted homogeneous polynomial  $w^2 - F_{10}(x,y,z)$ with at worst canonical singularities.
\end{proposition}

This sets up our convention for this paper where $S$ will be a surface defined by:
\[
S := Z(w^2 - F_{10}(x,y,z)) \subset \mathbb{P}(1,1,2,5)
\]
where $x,y$ is of degree 1, $z$ is degree 2, $w$ is degree 5 and $F_{10}$ a degree 10 weighted homogeneous polynomial. We will frequently write $F_{10}$ graded by $z$ with the form:
\[
F_{10}(x,y,z) = \sum_{i = 0}^5 q_i(x,y)z^i
\]
where $q_i(x,y) \in \mathbb{C}[x,y]_{10- 2i}$ are homogenous polynomial of degree $10-2i$.

\subsection{Related Moduli Problems}
\subsubsection{Moduli of Canonical Surfaces of type $K_S^2 = 1$ and $\chi(S) = 3$}

As $\chi(S) = 1 - q(S) + p_g(S)$, it should be expected that the moduli space of surfaces of type $K_S^2 = 1$ and $p_g(S) = 2$ contain components for various $q_(S)$ and as a result for various $\chi(S)$ also. This will turn out to not be the case due to the fact that from \cite[Exercise X.13.5]{Beauville96} the assumption of $K_S^2 = 1$ and $p_g(S) = 2$ implies that $q(S) = 0$. Thus, the moduli space in question that is being studied in this paper is $M_{(1,2,0)}$.

\subsubsection{$M_{(1,2,0)}$ and Moduli of Genus $2$ Fibrations on Surfaces}

It is known that, for canonical surfaces of type $K_S^2 = 1$ and $p_g(S) = 2$, the canonical map $|K_S|$ has a unique base point that when resolved produces a surface with a genus $2$ fibration over $\mathbb{P}^1$. Furthermore, from \cite[Lemma 2.1]{CatanesePignatelli2006}, we see that these surfaces are the only surfaces which are relatively minimal but not absolutely minimal. Thus, the moduli problem of these non-minimal surfaces that admit a relatively minimal genus $2$ fibrations is equivalent to the moduli problem of minimal surfaces, $S$, with $K_S = 1$ and $p_g(S) = 2$.

\begin{comment}
It can be observed, in \cite[Theorem 4.13]{CatanesePignatelli2006}, that the moduli of genus $2$ fibrations on surfaces are realized by the datum of:
\[
(B, V_i, \tau, \xi, w)
\]
where:
\begin{itemize}
	\item $B$ is a smooth curve.
	\item $V_1$ is a vector bundle of rank $2$ on $B$.
	\item $\tau$ is an effective divisor on $B$
	\item $\xi \in Ext_{\mathcal{O}_B}(\mathcal{O}_\tau, S^2(V_1)) \slash Aut_{\mathcal{O}_B}(\mathcal{O}_\tau)$ yields a vector bundle $V_2$.
	\item $w \in \mathbb{P}(H^0(B, \overline{\mathcal{A}}_6))$ where $\overline{\mathcal{A}}_6$ is determined by $\xi$.
\end{itemize}
It can be checked that in our situation that $B \cong \mathbb{P}^1$, $\tau = 0$ and $\deg(V_1) = 4$  as a vector bundle over $B = \mathbb{P}^1$. Then we see that $V_1 = \mathcal{O}_{\mathbb{P}^1}(1) \oplus \mathcal{O}_{\mathbb{P}^1}(3)$, $V_2 = Sym_2(V_1)$ and 
\[
\overline{\mathcal{A}}_6= \bigoplus_{k = -1}^5 \mathcal{O}_{\mathbb{P}^1}(2k)
\]
Thus the only parameters is this data is in $\overline{\mathcal{A}}_6$. This data is equivalent to the data of the weighted polynomial $F_{10}(x,y,z)$ used to define the canonical surface $S$ in $\mathbb{P}(1,1,2,5)$. This is a direct relation between the moduli of admissible $5$-tuples above with the moduli problem of canonical surfaces of general type with $K_S^2 = 1$ and $p_g(S) = 2$.
\end{comment}

\subsubsection{Non-reductive GIT quotient}

The approach of this paper is to reduce the problem into a situation where we can apply and take GIT quotients under a reductive group action. It is also possible to approach the same moduli problem by taking a non-reductive GIT quotient and references can be made to \cite{Kirwan09} for the relevant non-reductive GIT quotient.

\section{A Parameter Space of $S$ and Setting Up GIT}
This section simplfies the moduli problem by finding a unique surface in the isomophism class of $S$ which can be paramaterized by a space which admits an action by a reductive group. For the rest of the paper, our surfaces are of the form in proposition \ref{form1}, and is a hypersurface in $\mathbb{P}(1,1,2,5)$ defined by $w^2 - F_{10}(x,y,z)$, a weight homogneous polynomial of degree $10$. We then have the following properties of such surfaces and isomorphisms between surfaces.

\begin{lemma}[{\cite[Prop. 5, Lemma 12]{Wen21}}]
\label{initialForm}
Let $S$ be as above and assume that $S$ is surface with at worst canonical singularities, then:
\begin{itemize}
	\item $q_5 \neq 0$
	\item $q_0(x,y) \neq 0$ or $q_1(x,y) \neq 0$
	\item $\displaystyle \prod_{i = 0}^5 q_i(x,y)$ has at least $2$ linear factors.
\end{itemize}
\end{lemma}

\begin{lemma}[{cf. \cite[Prop. 3]{Wen21}}]
\label{form}
Let $S$ and $S'$ be isomorphic surfaces with the form above defined in $\mathbb{P}(1,1,2,5)$ then the isomorphism $\phi:S \rightarrow S'$ is induced by an automorphism of $\mathbb{P}(1,1,2,5)$.
\end{lemma}

\begin{proof}
The isomorphism $\phi: S \rightarrow S'$ induces an isomorphism for each $k \in \mathbb{Z}_{\geq 0}$,
\[
\phi_k: H^0(S, \omega_S^{\otimes k}) \rightarrow H^0(S', \omega_{S'}^{\otimes k})
\]
Furthermore, $H^0(S, \omega_S^{\otimes m}) \cong H^0(S,\mathcal{O}(m))$ and since $S$ is degree $10$, then for each $m \leq 9$ we get:
\[
H^0(\mathbb{P}(1,1,2,5), \mathcal{O}(m)) \xrightarrow{\sim} H^0(S, \mathcal{O}(m))
\]
The same is also true for $S'$. In particular for $0 \leq k = m \leq 9$, $\phi_k$ induces the morphism:
\[
\phi_k: H^0(\mathbb{P}(1,1,2,5), \mathcal{O}(k)) \rightarrow H^0(\mathbb{P}(1,1,2,5), \mathcal{O}(k))
\]
but these are the first $9$ graded parts of the graded ring defining $\mathbb{P}(1,1,2,5)$ since
\[
\mathbb{C}[x,y,z,w] = \bigoplus_{m = 0}^\infty H^0(\mathbb{P}, \mathcal{O}(m))
\]
Furthermore, we know the associated graded ring is generated in degree $5$. This implies that $\phi_k$, for $0 \leq k = m \leq 9$, extends to an automorphism of the graded ring:
\[
\phi_\infty: \bigoplus_{m = 0}^\infty H^0(\mathbb{P}(1,1,2,5), \mathcal{O}(m)) \rightarrow \bigoplus_{m = 0}^\infty H^0(\mathbb{P}(1,1,2,5), \mathcal{O}(m))
\]
which corresponds to an automorphism of $\mathbb{P}(1,1,2,5)$.
\end{proof}

\begin{lemma}
\label{finalForm}
Let $S$ be a canonical surface of the form:
\[
S := Z(w^2 - F_{10}(x,y,z)) \subset \mathbb{P}(1,1,2,5)
\]
then $S$ is isomorphic to an canonical surface, $\hat{S}$, define by the equation:
\[
w^2 - z^5 - \sum_{i = 0}^3 q_i(x,y)z^i
\]
where $q_i(x,y) \in \mathbb{C}[x,y]$ are homogeneous of degree $10-2i$.
\end{lemma}

\begin{proof}
Let $\displaystyle F_{10}(x,y,z) = \sum_{i = 0}^5 q_i(x,y)z^i$, then under the assumptions, we are in the situation of lemma \ref{initialForm} so $q_5(x,y) \neq 0$. We apply the following automorphism of $\mathbb{P}(1,1,2,5)$ defined by the action:
\begin{align*}
w &\mapsto w\\
z &\mapsto \frac{1}{\sqrt[5]{q_5(x,y)}} z + \frac{\zeta_8}{5q_5(x,y)} q_4(x,y)\\
x &\mapsto x\\
y &\mapsto y 
\end{align*}
where $\zeta_8 = e^{\frac{\pi i}{4}}$. The resulting surface will have $q_4(x,y) = 0$, which is the form described as $\hat{S}$.
\end{proof}

\begin{definition}
We say a canonical surface $S \subset \mathbb{P}(1,1,2,5)$ is in normal form if it is defined by the equation with the form:
\[
w^2 - z^5 - \sum_{i = 0}^3 q_i(x,y)z^i
\]
as described above.
\end{definition}

\begin{lemma}
Let $S$ and $\tilde{S}$ be surfaces in normal form and let $\phi: S \rightarrow \tilde{S}$ be an isomorphism. Then $\phi$ can be realized as an action of some unique $A \in GL(2,\mathbb{C})$ acting on the degree 1 graded part of the graded rings associated to $S$ and $\tilde{S}$.
\end{lemma}

\begin{proof}
Following lemma \ref{form}, the isomorphism is realized as an automorphism of $\mathbb{P}(1,1,2,5)$ defined by
\begin{align*}
w &\mapsto \alpha w + p(x,y,z)\\
z &\mapsto \beta z + r(x,y)\\
x &\mapsto ax + by\\
y &\mapsto cx + dy 
\end{align*}
where $\alpha \in \mathbb{C}^*$, $p(x,y)$ is a degree $5$ homogeneous polynomial, $r(x,y)$ is a degree $2$ homogeneous polynomial and
\[
A = 
\begin{pmatrix}
a & b \\
c & d
\end{pmatrix}\in GL(2, \mathbb{C})
\]
In both defining equations, of the surfaces in normal form, that there is no $w$ and $z^4$ monomial terms and so $p(x,y) = 0$ and $r(x,y) = 0$. Furthermore since the coefficient of $w^2$ and $z^5$ is $1$, this implies that $\alpha = \beta = 1$. Thus any isomorphism of these surfaces in normal form are completely understood by a unique $A \in GL(2,\mathbb{C})$ acting on the degree $1$ components of the graded ring. Uniqueness comes from the fact that if there were two $A,B \in GL(2,\mathbb{C})$ then the action of $A$ composed with the inverse action of $B$ would give $AB^{-1} = I$ and so $A = B$.
\end{proof}

From the above results, the problem of explciitely describing the  moduli space of canonical surfaces with $K_S^2 = 1$ and $p_g(S) = 2$ can be reduced to studying the action of $GL(2,\mathbb{C})$  on the parameters which are coefficients of $q_i(x,y)$ of surfaces in normal form. This allows for a GIT quotient which will be computed in two step by taking intermediate geometric quotients to reduce the problem down to a case of an $SL(2,\mathbb{C})$ action. This can be seen in the next two propositions.

\begin{proposition}
We can parameterize canonical surfaces in normal form by an open subset, $U \subset \mathbb{C}^{32}$
\end{proposition}

\begin{proof}
We consider the subspace of $H^0(\mathcal{O}_{\mathbb{P}(1,1,2,5)}(10))$ consisting of monomials that does not contain $w^2$, $w$, $z^4$ and with $z^5$ with coefficient $1$. This is exactly equations of the form:
\[
z^5 + \sum_{i = 0}^3 q_i(x,y)z^i
\]
where $q_k(x,y)$ are homogeneous polynomials of degree $10-2k$ and each choice will give a surface:
\[
w^2 - z^5 - \sum_{i = 0}^3 q_i(x,y)z^i
\]
with each $q_i$ having $11-2i$ degrees of freedom. This give a parameter space, $\mathbb{C}^{32}$, of such surfaces. Certainly, a general element of this paramater space is smooth, but as having at worst canonical singularities is also an open condition, from \cite[Remark 4.21]{KollarMori}, there is some open $U \subset \mathbb{CP}^{32}$ that parameterizes canonical surfaces in normal form.
\end{proof}

\begin{proposition}
The action of $GL(2,\mathbb{C})$ on the space of surfaces of normal form induces an action on $\mathbb{C}^{32}$ and, as a result, a linearized action of $SL(2,\mathbb{C})$ on the weighted projective space $\mathbb{P}(4^5, 6^7,8^9, 10^{11})$.
\end{proposition}

\begin{proof}
We see that the action of $GL(2,\mathbb{C})$ on normal surfaces does not act on $z^5$. Thus the $GL(2,\mathbb{C})$ action on surfaces in normal form induces and action on the coefficients of $q_i(x,y)z^i$ for $0 \leq i \leq 3$ which give a linear action on $\mathbb{C}^{32}$.

%%%%%%%%%%%%%%%%%%%%%%%%%%%%%%%%%%%%%%%%%%%%%%%%%%%%%%%%%
\begin{comment}
Now the action of $GL(2,\mathbb{C})$ acts on the points of $\mathbb{CP}^{32}$ which corresponds to equivalence classes of surfaces of the form:
\[
w^2 - \gamma ( q_5(x,y)z^5 + \sum_{i = 0}^3 q_i(x,y)z^i)
\]
or equivalently orbits of $w^2 - q_5(x,y)z^5 - \sum_{i = 0}^3 q_i(x,y)z^i$ under the action of:
\begin{align*}
w &\mapsto w\\
z &\mapsto \gamma^{\frac{1}{5}} z \\
x &\mapsto \gamma^{\frac{1}{10}} x\\
y &\mapsto \gamma^{\frac{1}{10}} y 
\end{align*}

commutes with the $(x,y,z)$-scaling automorphisms defined previously. This induces an action of $GL(2,\mathbb{C})$ on $\mathbb{CP}^{32}$ where each point corresponds to the orbits of surfaces in normal form under the action of $(x,y,z)$-scaling automorphisms of $\mathbb{P}$ of the form:
\begin{align*}
w &\mapsto w\\
z &\mapsto \gamma^2 z \\
x &\mapsto \gamma x\\
y &\mapsto \gamma y 
\end{align*}

\end{comment}
%%%%%%%%%%%%%%%%%%%%%%%%%%%%%%%%%%%%%%%%%%%%%%%%%%%%%%%

We have that $\mathbb{C}^* \hookrightarrow GL(2,\mathbb{C})$ is a closed normal subgroup and thus acts on $\mathbb{C}^{32} \setminus \{0\}$. The action of $\mathbb{C}^*$ induced by $GL(2,\mathbb{C})$ scales the monomials terms, where given $\gamma \in \mathbb{C}^*$:
\[
\gamma \cdot (x^iy^jz^k) = (\gamma x)^i(\gamma y)^j z^k =\gamma^{i + j}(x^iy^jz^k)
\]
As each coefficients (excluding $w^2$) of the normal form
\[
w^2 - z^5 - \sum_{i = 0}^3 q_i(x,y)z^i
\]
is a coordinate of $\mathbb{C}^{32}$. We have then the quotient is a weighted projective space:
\[
(\mathbb{C}^{32} \setminus \{0\}) \slash \mathbb{C}^* = \mathbb{P}(4^5, 6^7,8^9, 10^{11})
\]
where the $a_i^{b_i}$ coordinates in the weighted projective space means that there is a $b_i$ number of coordinates that are of degree $a_i$. Furthermore, the action of $GL(2,\mathbb{C})$ induces an action of
\[
PGL(2,\mathbb{C}) := GL(2,\mathbb{C}) \slash \mathbb{C}^* 
\] 
on $\mathbb{P}(4^5, 6^7,8^9, 10^{11})$.

As $SL(2,\mathbb{C})$ as a double cover of $PGL(2,\mathbb{C})$, this gives an action of $SL(2,\mathbb{C})$ on $\mathbb{P}(4^5, 6^7,8^9, 10^{11})$. Furthermore, $\mathcal{O}(120m)$ is very ample on $\mathbb{P}(4^5, 6^7,8^9, 10^{11})$ for some $m \in \mathbb{N}$, thus $|\mathcal{O}(120m)|$ defines an embedding into a projective space.  

The action of $SL(2,\mathbb{C})$ on $\mathbb{P}(4^5, 6^7,8^9, 10^{11})$ can be realized as an action on the graded ring associated to the weighted projective space. In particular, this turns out to be
\[
R := \bigoplus_k H^0(\mathbb{P}(4^5, 6^7,8^9, 10^{11}), \mathcal{O}(k))
\]
As automorphisms of the graded ring must restrict to linear maps on the graded components, then the action of $SL(2,\mathbb{C})$ is linearized on $H^0(\mathbb{P}(4^5, 6^7,8^9, 10^{11}),\mathcal{O}(120m))$.
\end{proof}

\section{Stability of Linearized $SL(2,\mathbb{C})$ action on $\mathbb{P}(4^5, 6^7,8^9, 10^{11})$ and GIT quotient}

This section investigates the linearized action of $SL(2,\mathbb{C})$-action on $\mathbb{P}(4^5, 6^7,8^9, 10^{11})$ and will show that the canonical surfaces in normal form will be stable under this action. This will result in a GIT quotients which compactifies $M_{(1,2,0)}$ as well as allowing us to obtain geometric information of the moduli space itself. 

The convention for this section will be that we are working over $\mathbb{P}(4^5, 6^7,8^9, 10^{11})$ and we will write $H^0(\mathcal{O}(k))$ as shorthand for $H^0(\mathbb{P}(4^5, 6^7,8^9, 10^{11}), \mathcal{O}(k))$. Also we will denote points in $\mathbb{P}(4^5, 6^7,8^9, 10^{11})$ as:
\[
(a_{(i,j)})) \in \mathbb{P}(4^5, 6^7,8^9, 10^{11})
\]
with $i + j = 10 - 2k$ for $k \in \{0,1,2,3\}$. This is meant to realize $a_{(i,j)}$ also as the coefficients of $x^iy^j$ in $q_{k}(x,y)$. 

\begin{proposition}
\label{initialTest}
Let $p \in \mathbb{P}(4^5, 6^7,8^9, 10^{11})$ correspond to the surface defined by the following equation:
\[
w^2 - z^5 - \sum_{i = 0}^3 q_i(x,y)z^i
\]
such that there exists some $k \in \{0,1,2,3\}$ with $q_k(x,y)$ having linear factors of multiplicities $< 5-k$ (resp. $\leq 5-k$), then $p$ is a (semi-)stable point under the linear action of $SL(2,\mathbb{C})$ on $H^0(\mathcal{O}(120m))$.
\end{proposition}

\begin{proof}
Up to change of coordinates, the $1$-parameter subgroups of $SL(2,\mathbb{C})$ can be realized in the form
\[
A_r := \left [
\begin{matrix} 
\mu^r & 0\\
0 & \mu^{-r}
\end{matrix} 
\right ]
\] 
for $r \geq 0$ acting on $x$ and $y$. Furthermore, the sections of $H^0(\mathcal{O}(120m))$ corresponds to the $120m$-th graded component of the graded ring:
\[
R = \bigoplus_k H^0(\mathcal{O}(k))
\]
such that $Proj (R) = \mathbb{P}(4^5, 6^7,8^9, 10^{11})$. As the ring is generated in degree $10$, sections of $H^0(\mathcal{O}(120m))$ are polynomial with monomial terms that are products of the degree $4,6,8$ and $10$ graded parts whose degree sum, counting multiplicities, equals $120m$. Thus, with an appropriate choice of basis the coordinates of $H^0(\mathcal{O}(120m))$ has the following form:
\[
\left ( a_{(i_3, j_3)}^{\gamma_3}a_{(i_2, j_2)}^{\gamma_2}a_{(i_1, j_1)}^{\gamma_1}a_{(i_0, j_0)}^{\gamma_0} \right )
\]
where $4\gamma_3 + 6\gamma_2 + 8\gamma_1 + 10\gamma_0 = 120m$ with $i_k,j_k, \gamma_k \geq 0$, $i_k + j_k = 10-2k$ and $k \in \{0,1,2,3\}$. Now the action of $A$ on the above coordinates gives the following diagonal action:
\[
\left ( \mu^{r(\gamma_3(2j_3 - 4) + \gamma_2(2j_2 - 6) + \gamma_1(2j_1 - 8) + \gamma_0(2j_0 - 10))} \prod_{n = 0}^3 a_{(i_n,j_n)}^{\gamma_n} \right )
\]
The (semi-)stable locus is must satisfy the following two conditions:
\begin{equation}
\label{one}
\min \left \{ \gamma_3(2j_3 - 4) + \gamma_2(2j_2 - 6) + \gamma_1(2j_1 - 8) + \gamma_0(2j_0 - 10) : {\prod a_{(i_n,j_n)}^{\gamma_n} \neq 0}  \right \} < ( \leq) 0
\end{equation}
\begin{equation}
\label{two}
\max \left \{ \gamma_3(2j_3 - 4) + \gamma_2(2j_2 - 6) + \gamma_1(2j_1 - 8) + \gamma_0(2j_0 - 10) : {\prod a_{(i_n,j_n)}^{\gamma_n} \neq 0} \right \} > (\geq ) 0
\end{equation}
Now we use the assumption that $q_k$ has linear factors with multiplicities $< 5-k$ (resp. $\leq 5-k$). This implies that some $k$:
\[
\min \{j_n\} < 5-k  \hspace{5mm} (resp. \leq 5-k)
\]
thus we can find a $a_{(i_k,j_k)} \neq 0$ with $j_k < 5-k $ (resp. $\leq 5-k$). Then we know that $a_{(i_k,j_k)}^\gamma$ is a coordinate in $H^0(\mathcal{O}(120m))$ with $\gamma = \frac{120m}{i_k + j_k}$ since $120m$ is a common multiple of the degrees. This implies that:
\[
\gamma(2j_k - (10 - 2k)) \in \left \{ \gamma_3(2j_3 - 4) + \gamma_2(2j_2 - 6) + \gamma_1(2j_1 - 8) + \gamma_0(2j_0 - 10) : {\prod a_{(i_n,j_n)} \neq 0}^{\gamma_n}  \right \}
\]
Since $j_k < 5-k $ (resp. $\leq 5-k$) then $2j_k - (10 - 2k) < 0$ (resp. $\leq 0$). Thus condition (\ref{one}) is satisfied. By symmetry we also have:
\[
\min \{i_k\} < 5-k  \hspace{5mm} (resp. \leq 5-k)
\]
and, since we are working with $q_k$ being homogeneous in two variables, this implies that for each $k$ there exists $a_{(i_k,j_k)} \neq 0$ with $j_k > 5-k$, and by the same arguments condition (\ref{two}) is also satisfied. This proves that if $p$ corresponds to a surface with $q_k(x,y)$ having linear factors with multiplicities $< 5-k$ (resp. $\leq 5-k$) then $p$ is a (semi-)stable point.
\end{proof}

The above lemma shows that most canonical surfaces will appear in the stable locus, since surfaces defined with distinct factors for $q_i(x,y)$ for an open subset of the parameter space. To show all canonical surfaces in normal form are stable points under the action of $SL(2,\mathbb{C})$ on $\mathbb{P}(4^5, 6^7,8^9, 10^{11})$ we need an analysis of Du Val type singularities.

\begin{proposition}
Surfaces defined with the form:
\[
w^2 - z^5 - \sum_{k = 0}^3 q_k(x,y)z^k
\]
are stable points under the linearized $SL(2,\mathbb{C})$ on $\mathbb{P}(4^5, 6^7,8^9, 10^{11})$ by $H^0(\mathcal{O}(120m))$. 
\end{proposition}

\begin{proof}
From \cite[Section 4.25]{KollarMori}, the above surface has at worst canonical singularities implies that
\[
mult_{[x_0:y_0: 0]}(z^5 - \sum_{k = 0}^3 q_k(x,y)z^k ) \leq 3
\]
for all points $[x_0:y_0: 0] \in \mathbb{P}(1,1,2)$. This implies that each $q_k(x,y)$ with $k \in \{0,1,2\}$ cannot simultaneously have $x^{i_k}$ as a factor of $q_k(x,y)$ where $i_k > 3 - k$. By symmetrical arguments the same can be said for $y$, possibly for some other $k$. Thus there exists a $k \in \{0,1,2, 3\}$ such that $a_{(i,j)} \neq 0$ with $i + j = 10 - 2k$ and $i \leq 3 - k$ which implies $j \geq 7 - k >5-k$. Using the above arguments, of proposition \ref{initialTest}, we see that condition (\ref{one}) is satisfied and by the same symmetrical arguments condition (\ref{two}) is also satisfied. Thus, the points corresponding to the canonical surfaces in normal form are stable points.
\end{proof}

\begin{remark}
From proposition \ref{initialTest}, we see not all stable points corresponds to canonical surfaces. We can see this with the following surface of normal form:
\[
w^2 - z^5 - \sum_{i = 0}^3 q_i(x,y)z^i
\]
where each of the $q_k(x,y)$ has a common linear factor of multiplicity $4-k$ with all other linear factors being distinct. Then we have that for some $[x_0:y_0: 0] \in \mathbb{P}(1,1,2)$ that
\[
mult_{[x_0:y_0: 0]} \left (z^5 + \sum_{i = 0}^3 q_i(x,y)z^i \right) = 4
\]
which from \cite[Section 4.25]{KollarMori} is known to to have singularities worst than canonical.
\end{remark}

\begin{corollary}
There is a GIT quotient, $\overline{M}_{(1,2,0)}$, of $\mathbb{P}(4^5, 6^7,8^9, 10^{11})$ by $SL(2,\mathbb{C})$ which compactifies $M_{(1,2,0)}$.
\end{corollary}

\begin{theorem}
\label{main}
$M_{(1,2,0)}$ is irreducible, of dimension 28 and $\overline{M}_{(1,2,0)}$ is unirational.
\end{theorem}

\begin{proof}
As the GIT quotient is a geometric quotient on the stable locus, by counting dimensions we obtain $\dim (M_{(1,2,0)}) = 28$. As the stable locus is open in an irreducible variety, the image must also be irreducible. Lastly, there is a dominant rational map
\[
\mathbb{P}^{32} \rightarrow \mathbb{P}(4^5, 6^7,8^9, 10^{11}) \dashrightarrow \overline{M}_{(1,2,0)}
\]
which shows that $\overline{M}_{(1,2,0)}$ is unirational.
\end{proof}

\bibliographystyle{plain}
\bibliography{refs}{}

\end{document}